\documentclass[11pt,reqno]{amsart}
\usepackage{amsfonts,amssymb,amsmath,amsthm,adjustbox}
\usepackage[margin=1.25in]{geometry}
\usepackage[hidelinks]{hyperref}
\usepackage{enumerate}
\usepackage{microtype}
\usepackage{mathtools}
\usepackage{cite}
\usepackage{color}
\usepackage{tikz-cd}
\usepackage{float}
\usepackage[foot]{amsaddr}
\usepackage{orcidlink}

\def\Z{\mathbb{Z}}
\def\Q{\mathbb{Q}}

\def\F{\mathbb{F}}

\def\im{\operatorname{im}}

\def\rank{\operatorname{rank}}

\theoremstyle{plain}
\newtheorem{theorem}{Theorem}

\newtheorem{proposition}[theorem]{Proposition}
\theoremstyle{definition}

\tikzset{
    solid edge/.style={thick},
    hidden edge/.style={thick,dashed},
    face light/.style={fill=black!5},
    face medium/.style={fill=black!10},
    face dark/.style={fill=black!15},
    center dot/.style={circle, fill=black, inner sep=1.5pt},
    caption/.style={font=\small},
    dim label/.style={font=\small, inner sep=1.5pt}
}

\title{Rational Solids with Equal Surface Area and Volume}
\date{\today}

\author{Kate Finnerty}
\address{Kate Finnerty, Department of Mathematics, Harvard University, Cambridge, MA}

\author{Jacob Mayle}
\address{Jacob Mayle, Department of Mathematical Sciences, University of Delaware, Newark, DE}


\begin{document}

\begin{abstract}
We classify pairs consisting of a right square pyramid and a right square prism, each having rational base side length and rational height, for which the two solids have equal surface area and equal volume. We prove that, up to scaling by a common positive rational factor, there is a unique such pair. The problem reduces to determining the rational points on a genus \(2\) bielliptic curve whose Jacobian has rank \(2\). We accomplish this using a quadratic Chabauty computation followed by the Mordell--Weil sieve. The same analysis gives an analogous classification for right circular cones and right circular cylinders.
\end{abstract}

\maketitle

\section{Introduction} \label{S:1}

Consider the two solids shown below. On the left is a right square pyramid
with base side length \(10\) and height \(12\). On the right is a right square prism with base side length \(10\) and height \(4\). \smallskip

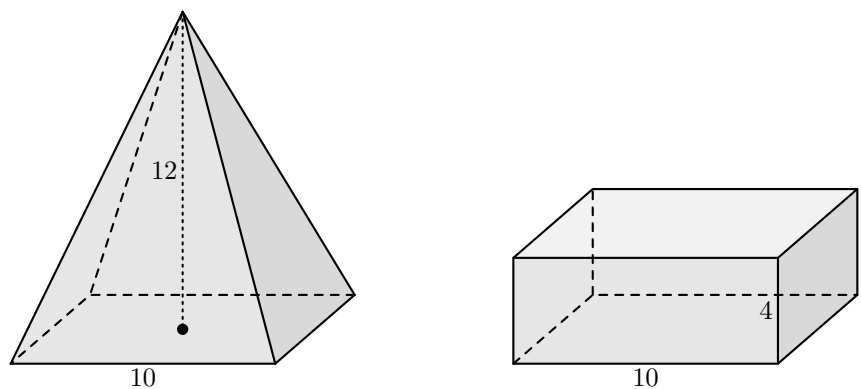
\begin{center}
\begin{tikzpicture}[
    scale=0.35,
    line join=round,
    line cap=round,
    every node/.style={font=\small}
]

\coordinate (A) at (0,0);
\coordinate (B) at (10,0);
\coordinate (C) at (13,2.6);
\coordinate (D) at (3,2.6);
\coordinate (O) at (6.5,1.3);
\coordinate (E) at (6.5,13.3);

\fill[face light]  (A)--(B)--(C)--(D)--cycle;
\fill[face medium] (A)--(B)--(E)--cycle;
\fill[face dark]   (B)--(C)--(E)--cycle;

\draw[solid edge] (A)--node[dim label,midway,below] {$10$} (B)--(C);
\draw[solid edge] (E)--(A);
\draw[solid edge] (E)--(B);
\draw[solid edge] (E)--(C);

\draw[hidden edge] (C)--(D)--(A);
\draw[hidden edge] (E)--(D);

\draw[hidden edge, dotted] (E)--node[dim label,midway,left] {$12$} (O);
\node[center dot] at (O) {};

\begin{scope}[shift={(19,0)}]

\coordinate (P0) at (0,0);
\coordinate (P1) at (10,0);
\coordinate (P2) at (13,2.6);
\coordinate (P3) at (3,2.6);

\coordinate (Q0) at (0,4);
\coordinate (Q1) at (10,4);
\coordinate (Q2) at (13,6.6);
\coordinate (Q3) at (3,6.6);

\fill[face medium] (P0)--(P1)--(Q1)--(Q0)--cycle;
\fill[face dark]   (P1)--(P2)--(Q2)--(Q1)--cycle;
\fill[face light]  (Q0)--(Q1)--(Q2)--(Q3)--cycle;

\draw[solid edge] (P0)--node[dim label,midway,below] {$10$} (P1)--(P2);
\draw[solid edge] (Q0)--(Q1)--(Q2)--(Q3)--cycle;
\draw[solid edge] (P0)--(Q0);
\draw[solid edge] (P1)--node[dim label,midway,left] {$4$} (Q1);
\draw[solid edge] (P2)--(Q2);

\draw[hidden edge] (P0)--(P3)--(P2);
\draw[hidden edge] (P3)--(Q3);
\end{scope}

\end{tikzpicture}
\end{center}

\medskip \noindent
The slant height of the pyramid is \[\sqrt{5^2 + 12^2} = 13.\]
We compute the surface area and volume of each solid:
\begin{alignat*}{2}
\operatorname{Area}(\text{Pyramid}) &= 10^2+4(\tfrac{1}{2}(10)(13))=360
\qquad &
\operatorname{Vol}(\text{Pyramid}) &= \tfrac{1}{3}(10^2)(12)=400 \\
\operatorname{Area}(\text{Prism}) &= 2(10^2)+4(10)(4)=360
\qquad &
\operatorname{Vol}(\text{Prism}) &= (10^2)(4)=400.
\end{alignat*}
The pyramid and the prism have equal surface area and equal volume. Moreover, all the dimensions used above to define these solids are rational. We call a right square pyramid or a right square prism \emph{rational} if its base side length and height are both positive rational numbers. Throughout, \emph{surface area} means total surface area, including every face.

The example above belongs to an infinite family: for each positive rational number \(\lambda\), consider the pyramid with base side length \(10\lambda\) and height \(12\lambda\), together with the prism with base side length \(10\lambda\) and height \(4\lambda\). These solids have common surface area \(360\lambda^2\) and common volume \(400\lambda^3\). In this article, we prove that these are the only such examples. 

\begin{theorem}\label{thm:main}
Suppose that a rational right square pyramid and a rational right square prism have equal surface area and equal volume. Then, up to scaling by a common positive rational factor, their dimensions are as follows:
\begin{center}
the pyramid has base side length \(10\) and height \(12\), and\\
the prism has base side length \(10\) and height \(4\).
\end{center}
Conversely, every pair obtained by scaling these dimensions by a common positive rational factor has equal surface area and equal volume.
\end{theorem}

In our definition of a rational right square pyramid, we do not require the slant height or the lateral edge length to be rational. The slant height is automatically rational for every pair satisfying the hypotheses of Theorem~\ref{thm:main}, as we observe in Section~\ref{S:2}. By contrast, the lateral edge length is irrational for the family in Theorem~\ref{thm:main}. Indeed, this length is
\[
\sqrt{(12\lambda)^2+(5\lambda)^2+(5\lambda)^2}
    =\sqrt{194}\,\lambda,
\]
which is irrational for every positive rational number \(\lambda\). Thus, there are no pairs satisfying the hypotheses of Theorem~\ref{thm:main}
if the lateral edge length is also required to be rational.

Questions concerning rational measurements in Euclidean geometry often lead to the arithmetic of algebraic curves. The classical congruent number problem, for example, asks which positive integers occur as the areas of right triangles with rational side lengths. Such integers are called \textit{congruent numbers}. For a fixed positive integer \(n\), the question of whether \(n\) is a congruent number is equivalent to determining whether the elliptic curve
\[ y^2=x^3-n^2x\] 
has positive Mordell--Weil rank; see, for example, \cite[Chapter~I]{MR1216136} or \cite[Section~1]{MR2757255}. A related problem in three dimensions concerns Euler bricks, which are rectangular prisms with integer edge lengths and integer face diagonals. Whether such a prism can also have an integer space diagonal, in which case it is called a perfect cuboid, is an open problem; see, for example, \cite[Exercise~1.8.17]{MR3929694}.

More recently, Hirakawa and Matsumura \cite{MR3860476, MR4474851} proved that, up to rational scaling,
there is a unique pair consisting of a rational right triangle and a rational isosceles triangle having equal perimeter and equal area. Their proof reduces the problem to determining the rational points on a genus \(2\) curve whose Jacobian has Mordell--Weil rank \(1\). Andrica and \c{T}urca\c{s} \cite{MR4203579} studied other pairs of symmetric invariants for rational right and isosceles triangles, and further equal-invariant problems for rational triangles appear in \cite{MR4346281, MR4484250}. Questions concerning a single solid whose surface area and volume are equal have been considered in \cite{Dart, RabagoTagle}.

Although the statement of Theorem~\ref{thm:main} is elementary, its proof leads naturally into modern arithmetic geometry. Equating the formulas for the surface areas and volumes of the two solids produces a system of polynomial equations. We show that the rational solutions to this system correspond to rational points on the algebraic curve \(C\) defined by
\[
y^2=x^6-9x^2+9.
\]
This curve has genus \(2\), is bielliptic over $\Q$, and its Jacobian has Mordell--Weil rank \(2\). Curves of genus at least \(2\) have only finitely many rational points by Faltings' theorem \cite{MR718935}. However, proving that one has determined all rational points can be challenging. For our curve $C$, the rank is equal to the genus, so the Mordell--Weil rank hypothesis for the classical Chabauty--Coleman method is not satisfied. Instead, we use the techniques described by Bianchi \cite{bianchi2020bielliptic} and Bianchi and Padurariu \cite{bianchi2022rational}; see also \cite[Sections~5--7]{MR3614022}. We first apply quadratic Chabauty to obtain a finite collection of \(p\)-adic points containing \(C(\Q)\), computed to finite precision, and then use the Mordell--Weil sieve of Bruin and Stoll \cite{MR2685127} to eliminate the \(p\)-adic points that do not arise from rational points. In this way, we prove the following proposition.

\begin{proposition}\label{prop:rational-points}
The rational points on \(C\) are
\[
C(\Q)=
\left\{
\infty^+,\infty^-,
(0,\pm3),
(\pm1,\pm1),
\left(\pm\tfrac{3}{2},\pm\tfrac{3}{8}\right)
\right\},
\]
where \(\infty^+\) and \(\infty^-\) are the two points at infinity on \(C\). In the last two expressions, the signs may be chosen independently.
\end{proposition}

Substituting the points in Proposition~\ref{prop:rational-points} back into the geometric equations gives the unique positive solution in Theorem~\ref{thm:main}. The argument used to prove Theorem~\ref{thm:main} also yields an analogous classification for right circular cones and right circular cylinders. We call either of these solids \emph{rational} if its radius and height are both positive rational numbers.

\begin{theorem}\label{thm:cone-cylinder}
Suppose that a rational right circular cone and a rational right circular cylinder have equal surface area and equal volume. Then, up to scaling by a common positive rational factor, their dimensions are as follows:
\begin{center}
the cone has radius \(5\) and height \(12\), and \\
the cylinder has radius \(5\) and height \(4\).
\end{center}
Conversely, every pair obtained by scaling these dimensions by a common positive rational factor has equal surface area and equal volume.
\end{theorem}

The figure below depicts the cone and cylinder described in Theorem~\ref{thm:cone-cylinder}. One can check that both solids have surface area \(90\pi\) and volume \(100\pi\). 
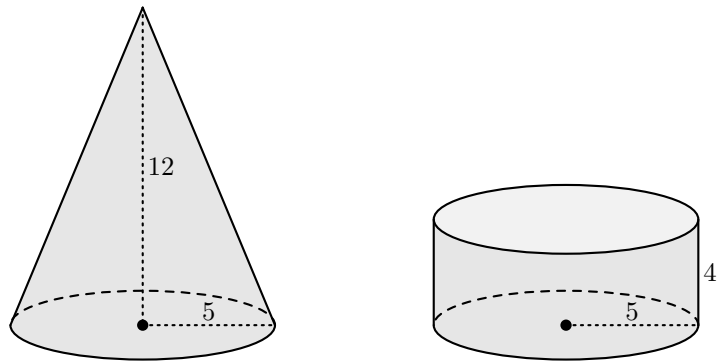
\medskip

\begin{center}

\begin{tikzpicture}[
    scale=0.35,
    line join=round,
    line cap=round,
    every node/.style={font=\small}
]

\begin{scope}[shift={(0,0)}]

\def\rx{5}
\def\ry{1.3}
\def\H{12}

\coordinate (Obase) at (0,0);
\coordinate (Apex) at (0,\H);
\coordinate (L) at (-\rx,0);
\coordinate (R) at (\rx,0);

\fill[face medium] (0,0) ellipse[x radius=\rx,y radius=\ry];

\fill[face medium]
    (L) arc[start angle=180,end angle=360,x radius=\rx,y radius=\ry]
    -- (Apex) -- cycle;

\draw[solid edge] (Apex)--(L);
\draw[solid edge] (Apex)--(R);

\draw[hidden edge] (R) arc[start angle=0,end angle=180,x radius=\rx,y radius=\ry];
\draw[solid edge]  (L) arc[start angle=180,end angle=360,x radius=\rx,y radius=\ry];

\draw[hidden edge, dotted] (Apex)--node[dim label,midway,right] {$12$} (Obase);

\draw[hidden edge, dotted] (Obase)--node[dim label,midway,above] {$5$} (R);

\node[center dot] at (Obase) {};

\end{scope}

\begin{scope}[shift={(16,0)}]

\def\rx{5}
\def\ry{1.3}
\def\h{4}

\coordinate (Cbot) at (0,0);
\coordinate (Ctop) at (0,\h);
\coordinate (Lbot) at (-\rx,0);
\coordinate (Rbot) at (\rx,0);
\coordinate (Ltop) at (-\rx,\h);
\coordinate (Rtop) at (\rx,\h);

\fill[face medium]
    (Lbot)--(Ltop)
    arc[start angle=180,end angle=360,x radius=\rx,y radius=\ry]
    --(Rbot)
    arc[start angle=0,end angle=180,x radius=\rx,y radius=\ry]
    --cycle;

\fill[face medium] (0,0) ellipse[x radius=\rx,y radius=\ry];
\fill[face light] (0,\h) ellipse[x radius=\rx,y radius=\ry];

\draw[solid edge] (Lbot)--(Ltop);
\draw[solid edge] (Rbot)--node[dim label,midway,right] {$4$} (Rtop);

\draw[solid edge] (0,\h) ellipse[x radius=\rx,y radius=\ry];

\draw[hidden edge] (Rbot) arc[start angle=0,end angle=180,x radius=\rx,y radius=\ry];
\draw[solid edge]  (Lbot) arc[start angle=180,end angle=360,x radius=\rx,y radius=\ry];

\draw[hidden edge, dotted] (Cbot)--node[dim label,midway,above] {$5$} (Rbot);

\node[center dot] at (Cbot) {};

\end{scope}

\end{tikzpicture}
\end{center}

\medskip The remainder of the article is organized as follows. In Section~\ref{S:2}, we pass from the geometric problem to the genus \(2\) curve \(C\). In Section~\ref{S:3}, we determine \(C(\Q)\) using quadratic Chabauty and the Mordell--Weil sieve. In Section~\ref{S:4}, we translate this arithmetic result back into geometry to prove Theorems~\ref{thm:main} and~\ref{thm:cone-cylinder}. The computations in this article were carried out using \texttt{Magma} \cite{MR1484478} and \texttt{SageMath} \cite{sagemath}. The code is available in the GitHub repository \cite{projectcode} associated with this article: \url{https://github.com/anon123680/rationalsolids}.

\subsection*{Acknowledgments}

We thank Jeremy Rouse and \'Alvaro Lozano-Robledo for helpful comments on an earlier version of this article. We also thank the organizers of the CTNT 2026 Conference, which provided the setting in which the authors began this collaboration.

\subsection*{AI Disclosure}

ChatGPT 5.6 was and Claude Opus 5 used to assist in writing the article and reviewing computer code. The authors take full responsibility for the content of the article.

\section{Reduction to a genus 2 curve} \label{S:2}

We now move from the geometric problem posed in Theorem~\ref{thm:main} to the curve \(C\) appearing in Proposition~\ref{prop:rational-points}. Scaling all the dimensions of the two solids by the same positive rational factor preserves the equalities of their surface areas and volumes. Since the base side length of the prism is positive and rational, we may therefore normalize this length to be \(1\).

Let \(x\) and \(h\) denote the base side length and height of the pyramid, respectively, and let \(s\) denote its slant height. Let \(b\) denote the height of the prism. The surface area of the prism is \(2+4b\) and the surface area of the pyramid is \(x^2+2xs\). Thus the equality of surface areas gives
\begin{equation}\label{E:Area}
2+4b=x^2+2xs.
\end{equation}
Similarly, the equality of volumes gives
\begin{equation}\label{E:Vol}
b=\frac{x^2h}{3}.
\end{equation}
Finally, the slant height of the pyramid satisfies
\begin{equation} \label{E:Slant}
    s^2=h^2+\frac{x^2}{4}.
\end{equation}

By hypothesis, \(x,h\), and \(b\) are all rational. Solving \eqref{E:Area} for \(s\) and then using \eqref{E:Vol}, we obtain
\begin{equation}\label{E:s-in-terms-of-x-h}
s
=\frac{2+4b-x^2}{2x}
=\frac{1}{x}+\frac{2b}{x}-\frac{x}{2}
=\frac{1}{x}+\frac{2xh}{3}-\frac{x}{2}.
\end{equation}
In particular, \(s\) is rational. This proves the assertion made after Theorem~\ref{thm:main} that the slant height is automatically rational for every pair under consideration.

Substituting \eqref{E:s-in-terms-of-x-h} into \eqref{E:Slant} gives
\[
\left(\frac{1}{x}+\frac{2xh}{3}-\frac{x}{2}\right)^2
    =h^2+\frac{x^2}{4}.
\]
After clearing denominators and collecting the terms involving \(h\), this becomes
\begin{equation}\label{E:quadratic-h}
(4x^4-9x^2)h^2+(-6x^4+12x^2)h+9-9x^2=0.
\end{equation}
We regard the left-hand side as a polynomial in \(h\). Notice that for \(x > 0\), the leading coefficient vanishes if and only if \(x=3/2\). Substituting this value into \eqref{E:quadratic-h} gives
\[
-\frac{27}{8}h-\frac{45}{4}=0,
\]
and hence
\(h=-10/3\), contrary to the geometric requirement that \(h\) be positive. We may therefore assume that $x \neq 3/2$, so \eqref{E:quadratic-h} is a genuine quadratic equation in \(h\).

Since this quadratic has rational coefficients and has a rational root \(h\), its discriminant must be a rational square. The discriminant is
\[
\Delta
=(-6x^4+12x^2)^2 - 4(4x^4-9x^2)(9-9x^2)
=36x^2(x^6-9x^2+9).
\]
Since $36x^2=(6x)^2$ is a nonzero rational square, there must exist \(y \in \Q\) such that
\begin{equation}
    \label{E:C_eq}
    y^2=x^6-9x^2+9.
\end{equation}
Thus, every geometric solution, subject to the normalization we have chosen, determines a rational point on the genus \(2\) curve \(C\) defined by \eqref{E:C_eq}. 

Conversely, given an affine rational point $(x,y) \in C(\Q)$ with \(x > 0\) and \(x \neq 3/2\), define
\begin{equation*}
h=\frac{3(x^3-2x+y)}{x(4x^2-9)},\qquad
b=\frac{x^2h}{3},\qquad
s=\frac1x+\frac{2xh}{3}-\frac{x}{2}.
\end{equation*}
The first formula is obtained from the quadratic formula applied to \eqref{E:quadratic-h}; changing \(y\) to \(-y\) gives the other root. The values of \(b\) and \(s\) are then forced by \eqref{E:Vol} and \eqref{E:s-in-terms-of-x-h}. Thus, we obtain a normalized geometric solution precisely when \(h,b\), and \(s\) are all positive. For later reference, the point \((1,-1)\in C(\Q)\) gives the geometric solution
\begin{equation*}
x=1,\qquad h=\frac65,\qquad b=\frac25,\qquad s=\frac{13}{10},
\end{equation*}
whereas the point \((1,1)\) gives a degenerate solution with \(h=0\).

The cone-and-cylinder problem leads to exactly the same normalized equations. Indeed, let \(x\) denote the diameter of the cone and normalize the cylinder so that its diameter is \(1\). One could work with the radius, but we find it slightly cleaner to use the diameter. The respective base areas are \(\pi x^2/4\) and \(\pi/4\), and their base circumferences are \(\pi x\) and \(\pi\). After the common factor \(\pi/4\) is canceled from the surface area and volume equations, the preceding analysis applies without modification.

\section{Rational points on the genus 2 curve} \label{S:3}

The polynomial \(x^6-9x^2+9\) is squarefree, with discriminant
\[
\operatorname{disc}(x^6-9x^2+9) = -2^6 \cdot 3^{14}.
\]
Thus the curve $C$ defined in \eqref{E:C_eq} is smooth of genus $2$ and has good reduction away from $2$ and $3$. We consider $C$ as a smooth projective curve. Since the defining polynomial has even degree
and a square leading coefficient, $C$ has two rational points at infinity, which we denote by \(\infty^+\) and \(\infty^-\). A naive point search using a bound of $10^4$ finds the twelve rational points listed in Proposition~\ref{prop:rational-points}. The challenge that remains is to prove that there are no others.

\subsection{Bielliptic structure and the Mordell--Weil group} \label{S:3.1}

A curve is \emph{bielliptic} if it admits a degree \(2\) map to an elliptic curve. As explained by Wetherell \cite[Section~4.1]{wetherell97thesis}, a genus \(2\) curve over $\Q$ is bielliptic if and only if it admits an even sextic model of the form
\begin{equation*}
    y^2=a_6x^6+a_4x^4+a_2x^2+a_0
\end{equation*}
with coefficients in $\Q$. In the present case, \eqref{E:C_eq} already has this form. The involutions \((x,y)\mapsto(-x,y)\) and \((x,y)\mapsto(-x,-y)\) give degree \(2\) quotient maps to elliptic curves. We use the following global minimal Weierstrass models for the two quotient curves, together with their quotient maps:
\begin{alignat*}{2}
E_1&: \, y^2=x^3-9x+9, \qquad &\phi_1(x,y)&=(x^2,y) \\
E_2&: \, y^2=x^3-27x+27, \qquad &\phi_2(x,y)&=\left(\frac9{x^2}-3,\frac{9y}{x^3}\right).
\end{alignat*}
These maps extend uniquely across $x = 0$ and the points at infinity to morphisms between the corresponding smooth projective curves.

Let $J = \operatorname{Jac}(C)$ be the Jacobian of $C$. For \(i=1,2\), let
\[
\phi_i^* \colon H^0(E_i,\Omega^1) \longrightarrow H^0(C,\Omega^1)
\]
denote pullback of regular differentials. The images of \(\phi_1^*\) and
\(\phi_2^*\) are linearly independent and together span
\(H^0(C,\Omega^1)\). In this way, the pair \((\phi_1,\phi_2)\) induces a
rational isogeny $J \to E_1 \times E_2$. One can check that both $E_1$ and $E_2$ have rank $1$, so
\[
\rank J = \rank E_1 + \rank E_2 = 1 + 1 = 2.
\]
In fact, we verify with the accompanying \texttt{Magma} code that
\begin{equation}\label{E:MW-group}
J(\Q)
=
\langle D_1,D_2\rangle\oplus\langle T\rangle
\cong \Z^2\oplus\Z/3\Z,
\end{equation}
where \(D_1=[\infty^+-\infty^-]\) and \(D_2=[(-1,-1)-\infty^-]\) form a basis for the free part of \(J(\Q)\) and \(T\) is a rational point of order \(3\). We use these generators for the remainder of the article. In our computations, we differentiate between $\infty^+$ and $\infty^-$ as the points at which the rational function $\frac{y}{x^3}$ equals positive and negative one, respectively.

\subsection{Applying quadratic Chabauty} \label{S:3.2}

Quadratic Chabauty builds on the classical method of Chabauty and Coleman \cite{chabauty1941points,coleman1985effective}. We first recall the classical setup. Let $C/\Q$ be a curve of genus $g \geq 2$ and Jacobian $J$. Fix a rational base point \(P_0\in C(\Q)\), and let
\[
\iota \colon C\longrightarrow J,
\qquad
P\longmapsto[(P)-(P_0)]
\]
be the Abel--Jacobi map.

\begin{theorem}[Chabauty, {\cite{chabauty1941points}}]\label{thm:classical-chabauty}
Let \(p\) be a prime of good reduction for \(C\). If \(\rank J(\Q)<g\), then the intersection of \(\iota(C(\Q_p))\) with the \(p\)-adic closure of \(J(\Q)\) in \(J(\Q_p)\) is finite.
\end{theorem}

Coleman made this method effective by developing a theory of \(p\)-adic integration \cite{coleman1985effective}. Integration gives a homomorphism
\[
J(\Q_p)
\longrightarrow
H^0(C_{\Q_p},\Omega^1)^\vee,
\qquad
D\longmapsto\left(\omega\longmapsto\int_D\omega\right).
\]
The Abel--Jacobi map is compatible with integration on the curve, as expressed by the following commutative diagram.
\[\begin{tikzcd}[column sep=4.5em]
	{C(\Q)} & {C(\Q_p)} \\
	{J(\Q)} & {J(\Q_p)} &  {H^0(C_{\Q_p},\Omega^1)^\vee\hspace{-.5cm}}
	\arrow[from=1-1, to=1-2]
	\arrow["\iota", from=1-1, to=2-1]
	\arrow["{\iota_p}", from=1-2, to=2-2]
	\arrow["{\omega\mapsto\int_{P_0}^\bullet\omega}", from=1-2, to=2-3]
	\arrow[from=2-1, to=2-2]
	\arrow["{D\mapsto(\omega\mapsto\int_D\omega)}"'{pos=0.5}, from=2-2, to=2-3]
\end{tikzcd}\]
When the rank of $J$ is less than the genus of $C$, linear algebra produces a nonzero regular differential whose integral vanishes on the $p$-adic closure of \(J(\Q)\). Pulling this integral back to \(C(\Q_p)\) gives a nonconstant locally analytic function whose zero set contains \(C(\Q)\). The finiteness of its zero set is the basis of the Chabauty--Coleman method.

Hirakawa and Matsumura \cite{MR3860476, MR4474851} use the bound on $|C(\Q)|$ produced by this method to determine the rational points on the genus $2$ curve that appears in their article. For our curve, however, the rank of \(J\) equals the genus of \(C\), so the classical Chabauty--Coleman method gives no conclusion. We therefore turn to Kim's nonabelian Chabauty framework \cite{kim2009Selmer}, which constructs nested sets
\[
C(\Q_p)_1\supseteq C(\Q_p)_2\supseteq C(\Q_p)_3\supseteq\cdots\supseteq C(\Q).
\]
Kim conjectured that these sets eventually become finite and, ultimately, equal \(C(\Q)\). The case of $n=2$ is known as \emph{quadratic Chabauty}. Balakrishnan and Dogra \cite{Balakrishnan2018QC1} showed that quadratic Chabauty applies under the rank condition
\[
\rank J < g - 1 + \rank \operatorname{NS}(J),
\]
where $\operatorname{NS}(J)$ denotes the N\'eron--Severi group of $J$ over $\Q$. In our case, the rational isogeny $J \to E_1 \times E_2$ implies that $\rank \operatorname{NS}(J) \geq 2$. The above inequality therefore holds.

Bianchi further developed an explicit quadratic Chabauty construction for genus \(2\) bielliptic
curves whose two elliptic quotients have rank \(1\) \cite{bianchi2020bielliptic}. Bianchi and Padurariu supplied a precision analysis and general implementation \cite{bianchi2022rational}; we use a modified version \cite{finnertycode} that was used in \cite{MR5070722}. The \texttt{README} of  \cite{finnertycode} details the differences between the versions. For a prime $p$ of good reduction for $C$ at which both elliptic curve quotients have ordinary reduction, the method produces an explicitly computable locally analytic function
\[
\widetilde{\rho}_p \colon C(\Q_p)\longrightarrow\Q_p
\]
and a finite set \(\Omega_p\subseteq\Q_p\) such that the rational points on \(C\) are contained in the finite set
\begin{equation*}
     A_p = \{z\in C(\Q_p):\widetilde{\rho}_p(z)\in\Omega_p\}.
 \end{equation*}
The function \(\widetilde{\rho}_p\) is defined using \(p\)-adic logarithms and local \(p\)-adic heights on \(E_1\) and \(E_2\). The set \(\Omega_p\) records the possible local height contributions at the primes of bad reduction.

We apply this construction for the curve given by \eqref{E:C_eq} at the good ordinary primes $p=5,7,$ and $11$, using working precision $O(p^{25})$. The computation recovers the twelve rational points of Proposition~\ref{prop:rational-points}, along with $296$ and $580$ extra $p$-adic candidates at $p=7$ and $p=11$, respectively. Since $\phi_j\circ\theta=\pm\phi_j$ for every automorphism $\theta\colon(x,y)\mapsto(\pm x,\pm y)$ of $C$, we have $\widetilde{\rho}_p\circ\theta=\widetilde{\rho}_p$, so each $A_p$ is stable under the group $G$ generated by these automorphisms. At $p=5$, we therefore record only one representative of each $G$-orbit, obtaining $42$ of them and hence $168$ $p$-adic candidates. 

For the subsequent Mordell--Weil sieve, the logarithmic computation determines the two free Mordell--Weil coefficients modulo \(p^4\); the accompanying \texttt{SageMath} code verifies that the recovered pair is determined to absolute precision at least $p^4$ and aborts otherwise. It remains to prove that none of the extra \(p\)-adic candidates arise from a rational point.

To extract Mordell--Weil congruence information from these points, we use the base point $P_0=\infty^-$ as before. By \eqref{E:MW-group}, every \(P\in C(\Q)\) has a unique expression
\begin{equation}\label{E:AJ-coordinates}
\iota(P)=n_1D_1+n_2D_2+eT,
\qquad
n_1,n_2\in\Z,\quad e\in\{0,1,2\}.
\end{equation}
The quadratic Chabauty computation gives congruence information about the coefficients \(n_1\) and \(n_2\). Indeed, for \(j\in\{1,2\}\), applying the \(p\)-adic logarithm gives
\[
\operatorname{Log}_p(\phi_j(P))-\operatorname{Log}_p(\phi_j(\infty^-))
=
n_1\operatorname{Log}_p(\phi_{j,*}D_1)
+n_2\operatorname{Log}_p(\phi_{j,*}D_2),
\]
where \(\phi_{j,*} \colon J \to E_j\) denotes the homomorphism induced by the quotient map \(\phi_j \colon C\to E_j\). Here we use that the logarithm vanishes on torsion. If an extra \(p\)-adic candidate \(z \in A_p\) comes from some \(P\in C(\Q)\), then evaluating the left-hand sides at \(z\) yields a necessary congruence for \((n_1,n_2)\) modulo \(p^4\). Since $(n_1,n_2)\in\Z^2$, a candidate whose recovered pair is \emph{not} $p$-adically integral cannot arise from a rational point and is discarded. Thus every extra \(p\)-adic candidate determines a candidate congruence class in the free part of \(J(\Q)\).

There is also compatibility between the computations at the different primes. The primes of bad reduction relevant to the local height calculation are \(2\) and \(3\), and the elements of \(\Omega_p\) are indexed by \(28\) possible pairs of rational numbers \((v_2,v_3)\) through expressions of the form
\[
v_2\log_p(2)+v_3\log_p(3).
\]
The 28 pairs are exactly the products of the four admissible values of $v_2$ with the seven admissible values of $v_3$, so no two classes coincide. The indexing pairs \((v_2, v_3)\) are independent of \(p\). Consequently, if points in \(A_5\), \(A_7\), and \(A_{11}\) arise from the same rational point of \(C\), then they must belong to the same local height class. Eight of the \(28\) classes contain no surviving candidate with coefficients that are $p$-adically integral for at least one of these three primes and therefore cannot contain an additional rational point.

For each of the remaining \(20\) classes, we combine the congruences obtained at \(p=5\) and \(p=11\) using the Chinese remainder theorem. With
\begin{equation}\label{E:M0}
M_0=5^4 \cdot 11^4=9150625,
\end{equation}
we obtain \(792\) possible pairs
\(
(n_1,n_2)
\) 
modulo $M_0$. Since \(M_0\) is relatively prime to \(3\), multiplication by \(M_0\) is an automorphism of the torsion subgroup \(\langle T \rangle\). Thus the torsion coordinate in \eqref{E:AJ-coordinates} disappears upon passage to \(J(\Q)/M_0J(\Q)\), and the \(792\) pairs above determine a set
\[
\mathcal C_{M_0}\subseteq J(\Q)/M_0J(\Q)
\]
of candidate classes. 

\subsection{Applying the Mordell--Weil sieve} \label{S:3.3} 

Suppose that \(P\in C(\Q)\) is not one of the twelve rational points already found. Then, for each of the quadratic Chabauty primes, \(P\) must correspond to one of the extra \(p\)-adic candidates considered above. Consequently, we obtain
\begin{equation*}
\pi_{M_0}(\iota(P))\in\mathcal C_{M_0},
\end{equation*}
where
\(
\pi_{M_0} \colon J(\Q)\longrightarrow J(\Q)/M_0J(\Q)
\)
is the quotient map. We eliminate these candidate classes using the Mordell--Weil sieve of Bruin and Stoll \cite{MR2685127}; see also \cite[Section~4.3]{bianchi2022rational}.

We briefly recall the sieve criterion. Let \(M\) be a positive integer, let
\[
\mathcal C_M\subseteq J(\Q)/MJ(\Q),
\]
and let \(S\) be a finite set of primes of good reduction for \(C\). For each \(\ell\in S\), reduction induces maps \(\alpha_{\ell,M}\colon
J(\Q)/MJ(\Q)
\longrightarrow
J(\F_\ell)/MJ(\F_\ell)\) and \(\beta_{\ell,M}\colon
C(\F_\ell)
\longrightarrow
J(\F_\ell)/MJ(\F_\ell)\). 
Writing
\[
\alpha_S=\prod_{\ell\in S}\alpha_{\ell,M},
\qquad
\beta_S=\prod_{\ell\in S}\beta_{\ell,M},
\]
the compatibility of the Abel--Jacobi map with reduction is expressed by the following commutative diagram.
\[
\begin{tikzcd}
C(\Q)
    \arrow[r,"\pi_M\circ\iota"]
    \arrow[d,"\operatorname{red}_S"']
&
J(\Q)/MJ(\Q)
    \arrow[d,"\alpha_S"]
\\
\displaystyle\prod_{\ell\in S}C(\F_\ell)
    \arrow[r,"\beta_S"']
&
\displaystyle\prod_{\ell\in S}
J(\F_\ell)/MJ(\F_\ell).
\end{tikzcd}
\]
It follows that a class \(c\in\mathcal C_M\) can contain the Abel--Jacobi image of a rational point only if
\[
\alpha_S(c)\in
\beta_S\left(
\prod_{\ell\in S}C(\F_\ell)
\right).
\]
Equivalently, a candidate is eliminated as soon as this necessary local condition fails at one sieve prime. In particular, no rational point can map into \(\mathcal C_M\) if
\begin{equation*}
\alpha_S(\mathcal C_M)
\cap
\beta_S\left(
\prod_{\ell\in S}C(\F_\ell)
\right)
=
\emptyset.
\end{equation*}

In the implementation of the algorithm, we work with the image $G_\ell=\im(J(\Q)\rightarrow J(\mathbb{F}_\ell))$ and the quotient $G_\ell/MG_\ell$, intersecting the reduced curve with $G_\ell$ first. The reduction of a rational point lies in this image and the map $G_\ell/MG_\ell\rightarrow J(\F_\ell)/MJ(\F_\ell)$ is not necessarily injective, so the criterion is valid and in general finer than the classical sieve described above.

Returning to the present problem, we apply this criterion to the candidates obtained in the quadratic Chabauty analysis using primes $5, 7,$ and $11$, keeping the \(20\) local height classes separate. We choose the set of sieving primes $S$ by the criterion of \cite[Section~4.3]{bianchi2022rational}: primes $\ell\leq10^4$ of good reduction for which $|(J(\F_\ell)/MJ(\F_\ell))|/(\ell+1) > 2$. We use the set computed for modulus $M_0$ as defined in \eqref{E:M0}. For \(19\) of our \(20\) local height classes, reduction at the good primes \(29\), \(61\), \(1279\), and \(3547\) eliminates every candidate modulo \(M_0\). For the remaining class, the sieve leaves four candidates in $J(\Q)/M_0J(\Q)$. To eliminate these, we refine the modulus to
\[
M_1=3M_0.
\]
Each class in \(J(\Q)/M_0J(\Q)\) has \(27\) lifts to \(J(\Q)/M_1J(\Q)\): there are three possible lifts of each of the two free coordinates and, since \(3\mid M_1\), three possible torsion coordinates. Thus the four surviving classes give \(4\cdot27=108\) candidate classes modulo \(M_1\). Applying the Mordell--Weil sieve at the good primes \(5\) and \(7\) eliminates all \(108\) candidates.

Thus, none of the extra \(p\)-adic candidates obtained from quadratic Chabauty at \(p=5\), \(7\), or \(11\) can arise from a rational point. It follows that the twelve rational points already found are all the rational points on \(C\). This proves Proposition~\ref{prop:rational-points}.

\section{Proofs of Theorems~\ref{thm:main} and~\ref{thm:cone-cylinder}} \label{S:4}

We now combine the work of Sections~\ref{S:2} and~\ref{S:3} to complete the proofs of our two theorems.

\begin{proof}[Proof of Theorem~\ref{thm:main}]
As in Section~\ref{S:2}, we normalize the base side length of the prism to be \(1\). A geometric solution gives an affine point \((x,y)\in C(\Q)\) with \(x>0\). By Proposition~\ref{prop:rational-points}, the only positive values of \(x\) that can occur are
\[
x=1
\quad\text{and}\quad
x=\frac32.
\]
In Section~\ref{S:2}, we noted that the value \(x=3/2\) gives \(h=-10/3\), so it cannot yield a
geometric solution. Similarly, the point \((1,1)\) gives \(h=b=0\). On the other hand, the point \((1,-1)\) gives the nondegenerate solution
\[
h=\frac65,\quad b=\frac25,\quad s=\frac{13}{10}.
\]
This is the unique normalized solution for which all defining dimensions are positive. If the original prism has base side length \(a\), restoring the scale gives
\[
\text{pyramid: base side }a,\ \text{height }\frac{6a}{5},
\qquad
\text{prism: base side }a,\ \text{height }\frac{2a}{5}.
\]
Writing \(a=10\lambda\) gives the dimensions stated in the theorem. The converse follows by substituting these dimensions into the formulas for surface area and volume.
\end{proof}

\begin{proof}[Proof of Theorem~\ref{thm:cone-cylinder}]
As observed at the end of Section~\ref{S:2}, after normalizing the diameter of the cylinder to be \(1\), the cone-and-cylinder problem gives exactly the same equations \eqref{E:Area}--\eqref{E:Slant}. Hence the unique normalized positive solution has cone diameter \(1\), cone height \(6/5\), cylinder diameter \(1\), and cylinder height \(2/5\). Restoring the scale and writing the common diameter as \(10\lambda\), the cone and cylinder both have radius \(5\lambda\), while their heights are \(12\lambda\) and \(4\lambda\), respectively. This is the family stated in Theorem~\ref{thm:cone-cylinder}.
\end{proof}

\bibliographystyle{amsplain}
\bibliography{references}

\end{document}